\documentclass[a4paper,12pt]{amsart}
\usepackage{epsfig}
\usepackage{amsthm,amsfonts}
\usepackage{amssymb,graphicx,color}
\usepackage[all]{xy}
\usepackage{cancel} %\mathbb{C}ancel{}, \bcancel{}, \xcancel{}
\usepackage{verbatim}
\usepackage{hyperref}
\usepackage{enumitem}

\usepackage[a4paper,top=2.5cm,bottom=2.5cm,left=3cm,right=3cm]{geometry}
\newtheorem{theorem}{Theorem}[section]

\newtheorem{corollary}[theorem]{Corollary}
\newtheorem{proposition}[theorem]{Proposition}

\theoremstyle{definition}
\newtheorem{definition}[theorem]{Definition}
\newtheorem{example}[theorem]{Example}

\theoremstyle{remark}

\numberwithin{equation}{section}

\begin{document}

% \title[short text for running head]{full title}
\title[Homeomorphisms that preserve multiplicity and tangent cones]{Some classes of homeomorphisms that preserve multiplicity and tangent cones}

%    Only \author and \address are required; other information is
%    optional.  Remove any unused author tags.

%    author one information
% \author[short version for running head]{name for top of paper}
\author[J. Edson Sampaio]{J. Edson Sampaio}
\address[J. Edson Sampaio]{(1) Departamento de Matem\'atica, Universidade Federal do Cear\'a,
	      Rua Campus do Pici, s/n, Bloco 914, Pici, 60440-900, 
	      Fortaleza-CE, Brazil. E-mail: {\tt edsonsampaio@mat.ufc.br} \newline 
	      (2) BCAM - Basque Center for Applied Mathematics,
	      Mazarredo, 14 E48009 Bilbao, Basque Country - Spain.   
	      E-mail: {\tt esampaio@bcamath.org}
}
% \mathbb{C}urraddr{}
% \email{}
\thanks{The author was supported by the ERCEA 615655 NMST Consolidator Grant and also by the Basque Government through the BERC 2018-2021 program and by the Spanish Ministry of Science, Innovation and Universities: BCAM Severo Ochoa accreditation SEV-2017-0718.
}
\dedicatory{Dedicated to Professor L\^e D\~ung Tr\'ang on the occasion of his 70th birthday}

%    The 2010 edition of the Mathematics Subject Classification is
%    the current definitive version.
\keywords{Multiplicity - Tangent cones - Zariski's Problems}
\subjclass[2010]{Primary MSC 14B05 - MSC 32S50}

\date{}

\begin{abstract}
In this paper we present some applications of A'Campo-L\^e's Theorem and we study some relations between Zariski's Questions A and B. It is presented some classes of homeomorphisms that preserve multiplicity and tangent cones of complex analytic sets. Moreover, we present a class of homeomorphisms that has the multiplicity as an invariant when we consider right equivalence and this class contains many known classes of homeomorphisms that preserve tangent cones. In particular, we present some effective approaches to Zariski's Question A. We show a version of these results looking at infinity. Additionally, we present some results related with Nash modification and Lipschitz Geometry. 
\end{abstract}

\maketitle

\section{Introduction}

In 1973, N. A'Campo in \cite{Acampo:1973} and L\^e D. T. in \cite{Le:1973} proved separately the following result.
\begin{theorem}[A'Campo-L\^e's Theorem]
Let $f,g\colon(\mathbb{C}^n,0)\to (\mathbb{C},0)$ be two complex analytic functions. Suppose that $V(f)$ is smooth at 0. If there is a homeomorphism $\varphi\colon(\mathbb{C}^n,V(f),0)\to (\mathbb{C}^n,V(g),0)$, then $V(g)$ is also smooth at 0.
\end{theorem}

The main aim of this paper is to give some applications of this theorem to Zariski multiplicity question and to Lipschitz Geometry (see Section \ref{main_results}). Thus, we are going to describe these applications.

Let $f\colon (\mathbb{C}^n,0)\to (\mathbb{C},0)$ be the germ of a reduced holomorphic function at the origin. We recall that \emph{the multiplicity of} $V(f)$ at the origin, denoted by $m(V(f),0)$, is defined as follows: we write
$$f=f_m+f_{m+1}+\cdots+f_k+\cdots$$ where each $f_k$ is a homogeneous polynomial of degree $k$ and $f_m\neq 0$. Then, $m(V(f),0):= m.$
In this case, the tangent cone of $V(f)$ at the origin is $C(V(f),0)=V(f_m)$.

In 1971, O. Zariski in \cite{Zariski:1971} proposes many questions and the most known among them is the following.
\begin{enumerate}[leftmargin=0pt]
\item[]{\bf Question A.} Let $f,g\colon(\mathbb{C}^n,0)\to (\mathbb{C},0)$ be two complex analytic functions. If there is a homeomorphism $\varphi\colon(\mathbb{C}^n,V(f),0)\to (\mathbb{C}^n,V(g),0)$, is it true that $m(V(f),0)=m(V(g),0)$?
\end{enumerate}
In order to know more about this question, you can see \cite{Eyral:2007}.

Let us remark that since $V(f)$ is smooth at 0 if and only if $m(V(f),0)=1$, then A'Campo-L\^e's Theorem gives a positive answer to Question A when $m(V(f),0)=1$. 

Let $Bl_0(\mathbb{C}^n)=\{(x,[v])\in \mathbb{C}^n\times \mathbb{C}P^{n-1};x\wedge v=0\}$ and $\beta\colon Bl_0(\mathbb{C}^n)\to \mathbb{C}^n$ be the projection onto $\mathbb{C}^n$, where $x\wedge v=0$ means that there exists $\lambda \in \mathbb{C}$ such that $x=\lambda v$. If $f\colon(\mathbb{C}^n,0)\to (\mathbb{C},0)$ is a complex analytic function, we define $Bl_0(V(f))=\overline{\beta^{-1}(V(f)\setminus \{0\})}$ and $E_{0}(f)=Bl_0(V(f))\cap (\{0\}\times \mathbb{C}P^{n-1})$. 
Remark that $E_{0}(f)=\{0\}\times \mathbb{P}C(V(f),0)$, where $\mathbb{P}C(V(f),0)$ is the projectivized tangent cone of $V(f)$.

Thus, another question asked in \cite{Zariski:1971} was the following.
\begin{enumerate}[leftmargin=0pt]
\item[]{\bf Question B.} Let $f,g\colon(\mathbb{C}^n,0)\to (\mathbb{C},0)$ be two complex analytic functions. If there is a homeomorphism $\varphi\colon(\mathbb{C}^n,V(f),0)\to (\mathbb{C}^n,V(g),0)$, is there a homeomorphism $h\colon E_{0}(f)\to E_{0}(g)$ such that for each $p\in E_{0}(f)$ the germs $(Bl_0(V(f)), p)$ and $(Bl_0(V(g)),h(p))$ are homeomorphic?  
\end{enumerate}

 We would like to consider also the following versions of Questions A and B:
\begin{enumerate}[leftmargin=0pt]
\item [] Let $f,g\colon(\mathbb{C}^n,0)\to (\mathbb{C},0)$ be two reduced complex analytic functions. Suppose that there is a homeomorphism $\varphi\colon(\mathbb{C}^n,0)\to (\mathbb{C}^n,0)$ such that $f=g\circ \varphi$.\\
\item[]{\bf Question A'.} Is it true that $m(V(f),0)=m(V(g),0)$?\\
\item[]{\bf Question B'.} Suppose that $C(V(f),0)$ or $C(V(g),0)$ is a linear subspace. Is there a homeomorphism between $(C(V(f),0),0)$ and $(C(V(g),0),0)$ or between $E_{0}(f)$ and $E_{0}(g)$?
\end{enumerate}
 
Let me remark that by using Corollary in \cite{King:1978} together with Corollary 2 in \cite{Saeki:1989}, we obtain that, in the case of functions with isolated singularities, Question A' is equivalent to Question A. Moreover, even in the case of isolated singularities, Question A is still an open problem. However, J. Fern\'andez de Bobadilla in \cite{Bobadilla:2005} showed that Question B has a negative answer and in \cite{Bartolo:2010} the authors presented complex analytic functions $f,g\colon(\mathbb{C}^3,0)\to (\mathbb{C},0)$ with isolated singularities such that there is a homeomorphism $\varphi\colon(\mathbb{C}^3,0)\to (\mathbb{C}^3,0)$ satisfying $f=g\circ \varphi$, but there is no homeomorphism $\tilde h\colon(\mathbb{C}^3,0)\to (\mathbb{C}^3,0)$ such that $\tilde h(C(V(f),0))=C(V(g),0)$. In \cite{Bartolo:2010}, the authors used A'Campo-L\^e's Theorem to obtain that a such $\tilde h$ preserves the singular points of the tangent cones and, then, arrived in a contradiction. However, a homeomorphism $h\colon (C(V(f),0),0)\to (C(V(g),0),0)$ does not need to preserve singular points when $h$ is not a restriction of a homeomorphism $\tilde h\colon(\mathbb{C}^3,0)\to (\mathbb{C}^3,0)$. Here, we show that the example presented in \cite{Bartolo:2010} still gives a negative answer to Question B' when we do not require that $C(V(f),0)$ or $C(V(g),0)$ need to be a linear subspace (see Proposition \ref{ex_rel_inv}), though that the Question B' has a positive answer when $n=2$.

As a first application of A'Campo-L\^e's Theorem, we show that Questions A' and B' are related. In fact, we prove by using A'Campo-L\^e's Theorem the following.

\begin{theorem}\label{B_implies_A}
If Question B' has a positive answer then Question A' has a positive answer as well.
\end{theorem}

Another problem related with the above questions is to know if the relative multiplicities are topological invariants. To be more precise, let us define which are the relative multiplicities. Let $f\colon (\mathbb{C}^n,0)\to (\mathbb{C},0)$ be a reduced complex analytic function and 
$$f=f_m+f_{m+1}+\cdots+f_k+\cdots$$ 
is as before. We have the decomposition of $f_m$ in irreducible polynomials
$$f_m=h_1^{k_1}\cdots h_r^{k_r}.$$ 
Thus, we define the relative multiplicity of $V(f)$ (along of $V(h_i)$) by $k_{V(f)}(V(h_i))\colon =k_i$. In particular, we obtain that $m(V(f),0)=\sum k_{V(f)}(V(h_i)) m(V(h_i),0)$.

\begin{enumerate}[leftmargin=0pt]
\item[] Let $f,g\colon (\mathbb{C}^n,0)\to (\mathbb{C},0)$ be two reduced complex analytic functions. Let $X_1,...,X_r$ (resp. $Y_1,...,Y_s$) be the irreducible components of $C(V(f),0)$ (resp. $C(V(g),0)$). Suppose that there is a homeomorphism $\varphi\colon(\mathbb{C}^n,0)\to (\mathbb{C}^n,0)$ such that $f=g\circ \varphi$.\\
\item[]{\bf Question B$_{Rel}$.} Is there a bijection $\sigma\colon \{1,...,r\}\to \{1,...,s\}$ such that $k_{V(f)}(X_i)=k_{V(g)}(Y_{\sigma(i)})$, for all $i\in\{1,...,r\}$?\\
\item[]{\bf Question B'$_{Rel}$.} Suppose that $C(V(f),0)$ or $C(V(g),0)$ is a linear subspace. Is it true that $r=s=1$ and $k_{V(f)}(X_1)=k_{V(g)}(Y_{1})$?
\end{enumerate}
Question B$_{Rel}$ was also asked in January 2017 on Mathoverflow.net \cite{user103279} and it already was answered in many particular cases, for example, the answer is known to be yes in the following special cases:
\begin{itemize}
\item if $n=2$ in \cite{Zariski:1932};
\item if $\varphi$ is a real analytic diffeomorphism in \cite{KurdykaR:1989};
\item if $\varphi$ is a subanalytic bi-Lipschitz homeomorphism in \cite{Valette:2010};
\item if $\varphi$ is a bi-Lipschitz homeomorphism in \cite{FernandesS:2016};
\item if $\varphi$ is a blow-spherical homeomorphism in \cite{Sampaio:2017}.
\end{itemize}
So, as a second application of A'Campo-L\^e's Theorem, we show the following.

\begin{theorem}\label{relative_cor}
If Question B'$_{Rel}$ has a positive answer then Question A' also has a positive answer.
\end{theorem}

In Proposition \ref{ex_rel_inv}, we show that Question B$_{Rel}$ has a negative answer.

We present also an application of A'Campo-L\^e's Theorem to Lipschitz Geometry. Let us remind the following.
\begin{definition}\label{lipschitz function}
Let $X\subset\mathbb{R}^n$ and $Y\subset\mathbb{R}^m$. A mapping $f\colon X\rightarrow Y$ is called {\bf Lipschitz} if there exists $\lambda >0$ such that is
$$\|f(x_1)-f(x_2)\|\le \lambda \|x_1-x_2\|, \quad \mbox{for all }x_1,x_2\in X.$$ A Lipschitz mapping $f\colon X\rightarrow Y$ is called
{\bf bi-Lipschitz} if its inverse mapping exists and is Lipschitz. When there exists an bi-Lipschitz mapping $f\colon X\rightarrow Y$ we say that $X$ and $Y$ are bi-Lipschitz homeomorphic.
\end{definition}

In \cite{NaumannS:2011}, the authors presented several notions related with the notion of Lipschitz manifold and among them, they had the following.
\begin{definition}
A subset $M \subset \mathbb{R}^N$ is called a $k$-dimensional {\bf Lipschitz submanifold} if for every $x_0 \in M$ there exist an open set $U \subset \mathbb{R}^N$ and a bi-Lipschitz mapping $\varphi\colon U\to\varphi(U)\subset \mathbb{R}^N$ such that $x_0\in U$, $\varphi(M\cap U)=\{(x_1,...,x_N)\in \varphi(U);\, x_{k+1}=...=x_N=0\}$ and $\varphi(x_0)=0$.
\end{definition}
\begin{definition}
A subset $M \subset \mathbb{R}^N$ is called a $k$-dimensional {\bf parametric Lipschitz manifold} if for every $x_0 \in M$ there exist an open neighborhood $U \subset \mathbb{R}^N$ of $x_0$, an open set $W\subset \mathbb{R}^k$ and a bi-Lipschitz mapping $\varphi\colon W\to M\cap U$.
\end{definition}
Moreover, the authors mention in (\cite{NaumannS:2011}, page 11) that they could not prove or disprove that a parametric Lipschitz manifold is a Lipschitz submanifold. Here, by using A'Campo-L\^e's Theorem, we disprove this.

\begin{proposition}\label{ex_lip}
There are parametric Lipschitz manifolds that are not Lipschitz submanifolds.
\end{proposition}

In \cite{BirbrairG:2018}, the authors presented germs at the origin of real semialgebraic sets (of dimension two) $S_1,S_2\subset \mathbb{R}^3$ such that they are bi-Lipschitz homeomorphic but there is no bi-Lipschitz homeomorphism $\phi\colon (\mathbb{R}^3,0)\to(\mathbb{R}^3,0)$ such that $\phi(S_1)=S_2$.

Finally, in Section \ref{sec:Nash}, we consider one version more of Question B. In fact, we consider a version of Question B by considering Nash modifications instead of the blowing up of the origin. Let us be more precise with this.

If $\mathbb{K}$ is $\mathbb{R}$ or $\mathbb{C}$ and $X\subset \mathbb{K}^n$ is an analytic subset, we denote by ${\rm Sing}(X)$ to be the singular points of $X$ and by ${\rm Reg}(X):=X\setminus {\rm Sing}(X)$ to be the regular points of $X$.
\begin{definition}
Let $X\subset \mathbb{K}^n$ be an analytic subset of pure dimension $r$ and $ G_{r}^{n}(\mathbb{K})$ be the Grassmannian of $r$-planes in $\mathbb{K}^n$. We define a map $\nu \colon{\rm Reg}(X)\rightarrow X\times G_{r}^{n}(\mathbb{K})$, by $\nu (x):=(x,T_x X)$, where $T_x X$ is the tangent space of $X$ at $x$. The closure of the image of $\nu$, denoted by $X'_{Nash}$, is called the {\bf Nash modification of $X$} and the mapping induced by first projection $\eta_X\colon X'_{Nash}\to X$ is called the {\bf Nash mapping of $X$}.
\end{definition}
Thus, we have the following.
\begin{enumerate}[leftmargin=0pt]
\item[]{\bf Question B$_{Nash}$.}  Suppose that two germs of analytic sets $(X,0)$ and $(Y,0)$ are bi-Lipschitz homeomorphic. Is there a homeomorphism between $X'_{Nash}$ and $Y'_{Nash}$ sending $\eta_X^{-1}({\rm Sing}(X))$ over $\eta_Y^{-1}({\rm Sing}(Y))$?
\end{enumerate}
By proving in Proposition \ref{ne_cone} that the cone over a LNE subset of $\mathbb{S}^n$ is also LNE (see definition \ref{definition-ne}), we show the following.
\begin{proposition}\label{ex_Nash}
Question B$_{Nash}$ has a negative answer when we consider real analytic sets.
\end{proposition}

\section{Some applications of A'Campo-L\^e's Theorem}\label{main_results}

\subsection{Invariance of the multiplicity by right equivalence}\label{subsec:mult}
In this Subsection, among other things we prove that a positive answer to Question B' or B'$_{Rel}$ implies in a positive answer to Question A'.

Let me introduce a new class of homeomorphisms that essentially generalizes bi-Lipschitz equivalence, weak directional equivalence, differentiable equivalence and $C^1$ equivalence.
\begin{definition}
Let $\varphi\colon (\mathbb{C}^n,0)\to (\mathbb{C}^n,0)$ be a homeomorphism.
We say that $\varphi$ preserves linear tangent cones if for any complex analytic functions $f,g\colon (\mathbb{C}^n,0)\to (\mathbb{C},0)$ such that $f=g\circ\varphi$ and $C(V(f),0)$ or $C(V(g),0)$ is a linear subspace of $\mathbb{C}^n$, we have that there exists a homeomorphism $\phi$ between $(C(V(f),0),0)$ and $(C(V(g),0),0)$ or between $E_{0}(f)$ and $E_{0}(g)$.
\end{definition}
\begin{definition}
For each $n\in \mathbb{N}$, let $\mathcal{H}_n$ be the collection of all germs of homeomorphisms from $(\mathbb{C}^n,0)$ to $(\mathbb{C}^n,0)$ and let $\mathcal{S}\colon \mathcal{H}\to \mathcal{H}$ be the mapping given by $\mathcal{S}(\varphi)=\varphi\times id_{\mathbb{C}}$, where $\mathcal{H}:=\bigcup\limits_{n\in\mathbb{N}}\mathcal{H}_n$. We say that $\mathcal{T}\subset \mathcal{H}$ is an invariant set of $\mathcal{S}$ if $\mathcal{S}(\mathcal{T})\subset \mathcal{T}$.
\end{definition}
\begin{definition}\label{def_tg_hom}
Let $\mathcal{T}$ be a subset of $\mathcal{H}$.
We say that $\mathcal{T}$ preserves linear tangent cones if each element $\varphi\in \mathcal{T}$ preserves linear tangent cones. We say that $\mathcal{T}$ preserves weakly linear tangent cones if for each element $\varphi\in \mathcal{T}$ there exists a positive integer $m$ such that $\mathcal{S}^m(\varphi)$ preserves linear tangent cones.
\end{definition}

\begin{example}
If $Lip$ is the collection, for each $n\in\mathbb{N}$, of all germs of bi-Lipschitz homeomorphisms $h\colon (\mathbb{C}^n,0)\to (\mathbb{C}^n,0)$, then by Theorem 3.2 in \cite{Sampaio:2016}, it is easy to verify that $Lip$ preserves linear tangent cones and is an invariant set of $\mathcal{S}$.
\end{example}

\begin{theorem}\label{main_thm}
Let $f,g\colon (\mathbb{C}^n,0)\to (\mathbb{C},0)$ be complex analytic functions and let $\mathcal{T}\subset \mathcal{H}$ be an invariant set of $\mathcal{S}$. Suppose that there exists $\varphi\in \mathcal{T}$ such that $f=g\circ\varphi$. If $\mathcal{T}$ preserves weakly linear tangent cones, then ${\rm ord}_0 f={\rm ord}_0 g$.
\end{theorem}
\begin{proof}
Since $\mathcal{T}$ preserves weakly linear tangent cones, there exists a positive integer $m$ such that $\tilde\varphi:=\mathcal{S}^m(\varphi)$ preserves linear tangent cones.\\

\noindent{\bf Claim.} ${\rm ord}_0 f\geq {\rm ord}_0 g$.\\

Suppose that ${\rm ord}_0 f<{\rm ord}_0 g$. If $k={\rm ord}_0 f$, then by A'Campo-L\^e's Theorem, we have that $k>1$. Let us define $\tilde f, \tilde g\colon (\mathbb{C}^{n}\times\mathbb{C}^m,0)\to (\mathbb{C},0)$ by
$\tilde f(x,t)=f(x)+t_m^k$ and $\tilde g(x,t)=g(x)+t_m^k$, where $t=(t_1,...,t_m)$. Then, $C(V(\tilde g),0)=\{t_m=0\}=\mathbb{C}^{n+m-1}\times\{0\}$ and $\tilde f=\tilde g\circ\tilde \varphi$. Moreover, since $f_k:={\bf in} f\not \equiv 0$ and $k>1$, we get that $f_k+t_m^k$ cannot be a power of a linear form and, in particular, $C(V(\tilde f),0)=V(f_k+t_m^k)$ is not a linear subspace.

Since $\tilde f=\tilde g\circ\tilde \varphi$ and $\tilde\varphi$ preserves linear tangent cones, we have two possible cases:

\noindent{\it 1. There exists a homeomorphism $\tilde\phi\colon (V(f_k+t_m^k),0)\to (\mathbb{C}^{n+m-1}\times\{0\},0)$.}

 Thus, by Prill's Theorem (Theorem in \cite{Prill:1967}), $V(f_k+t_m^k)$ is a linear subspace of $\mathbb{C}^{n+m}$, which is a contradiction. Then, ${\rm ord}_0 f\geq {\rm ord}_0 g$. 

\noindent{\it 2. There exists a homeomorphism $\tilde\phi\colon E_{0}(\tilde f)\to E_{0}(\tilde g)$.} 

If $\dim E_{0}(\tilde f)\not=2$, by Corollary 2.12 in (\cite{Dimca:1992}, p. 145), $\mathbb{P}V(f_k+t_m^k)$ is a linear subspace of $\mathbb{C}P^{n+m-1}$, since $E_{0}(\tilde f)\cong \mathbb{P}V(f_k+t_m^k)$ and  $E_{0}(\tilde g)\cong \mathbb{C}P^{n+m-2}$. In particular, $C(V(\tilde f),0)$ is a linear subspace of $\mathbb{C}^{n+m}$, which is a contradiction. 
If $\dim E_{0}(\tilde f)=2$ and, in particular, $n+m=4$, then there exists a positive integer $r$ such that $\bar \varphi=\mathcal{S}^r(\tilde\varphi)$ preserves linear tangent cones, since $\mathcal{T}\subset \mathcal{H}$ is an invariant set of $\mathcal{S}$ and preserves weakly linear tangent cones. Thus, we define $\bar f, \bar g\colon (\mathbb{C}^4\times\mathbb{C}^{r},0)\to (\mathbb{C},0)$ by
$ \bar f(z,s)=\tilde f(z)+s_r^{k+1}$ and $\bar g(z,s)=\tilde g(z)+s_r^{k+1}$, where $s=(s_1,...,s_r)$. Therefore, $\bar f=\bar g\circ\bar \varphi$. Moreover, $C(V(\bar g),0)$ is the linear subspace $\{t_m=0\}\subset \mathbb{C}^{r+4}$, $C(V(\bar f),0)=V(f_k+t_m^k)\subset \mathbb{C}^{r+4}$ and as before $C(V(\bar f),0)$ cannot be a linear subspace. However, since $\bar\varphi$ preserves linear tangent cones and $\dim E_{0}(\bar f)=2+r\not=2$, by the same reason as before, $C(V(\bar f),0)$ is a linear subspace of $\mathbb{C}^{n+m+r}$, which is a contradiction. Then, in any case, ${\rm ord}_0 f\geq {\rm ord}_0 g$, which finish the proof of Claim. 

We need also to show that ${\rm ord}_0 f\leq {\rm ord}_0 g$. In order to do this, let us define $\mathcal{T}^{-1}=\{\psi^{-1};\,\psi\in \mathcal{T}\}$. Now, it is clear that $\mathcal{T}^{-1}$ is an invariant set of $\mathcal{S}$, $g=f\circ\varphi^{-1}$ and $\varphi^{-1}\in \mathcal{T}^{-1}$. Moreover, by definition, a homeomorphism $\psi$ preserves linear tangent cones if and only if $\psi^{-1}$ preserves linear tangent cones, as well. Therefore, by Claim, ${\rm ord}_0 g\leq {\rm ord}_0 f$, which finish the proof.
\end{proof}

As a direct consequence, we obtain Theorem \ref{B_implies_A}.

\begin{definition}
We say that a homeomorphism $\varphi\colon(\mathbb{C}^n,0)\to (\mathbb{C}^n,0)$ preserves weakly relative multiplicities if for any two reduced complex analytic functions $f,g\colon (\mathbb{C}^n,0)\to (\mathbb{C},0)$ such that $C(V(f),0)$ or $C(V(g),0)$ is a linear subspace and $f=g\circ \varphi$, then $C(V(f),0)$ and $C(V(g),0)$ are irreducible and $k_{V(f)}(C(V(f),0))=k_{V(g)}(C(V(g),0)$.
\end{definition}

\begin{theorem}\label{relative_thm}
Let $f,g\colon (\mathbb{C}^n,0)\to (\mathbb{C},0)$ be reduced complex analytic functions and let $\mathcal{T}\subset \mathcal{H}$ be an invariant set of $\mathcal{S}$. Suppose that any element of $\mathcal{T}$ preserves weakly relative multiplicities. If there exists $\varphi\in \mathcal{T}$ such that $f=g\circ\varphi$, then $m(V(f),0)=m(V(g),0)$.
\end{theorem}
\begin{proof}
Suppose that $m(V(f),0)={\rm ord}_0 f<{\rm ord}_0 g=m(V(g),0)$. If $k={\rm ord} (f)$, then, as before, by A'Campo-L\^e's Theorem, we have that $k>1$. Let $f_k$ be the homogeneous polynomial formed by the monomials of $f$ that have degree $k$. Let us define $\tilde f, \tilde g\colon (\mathbb{C}^{n+1},0)\to (\mathbb{C},0)$ by
$\tilde f(x,t)=f(x)+t^k$ and $\tilde g(x,t)=g(x)+t^k$. Then, $C(V(\tilde g),0)=\{t=0\}=\mathbb{C}^n\times\{0\}$. Moreover, since $f_k\not \equiv 0$ and $k>1$, we get that $f_k+t^k$ cannot be a power of a linear form and, in particular, $C(V(\tilde f),0)=V(f_k+t^k)$ is not a linear subspace.
Now, let us define $\tilde\varphi\colon (\mathbb{C}^{n+1},0)\to (\mathbb{C}^{n+1},0)$ by $\tilde\varphi(x,t)=(\varphi(x),t)$. We have that $\tilde\varphi\in \mathcal{T}$. Then, $C(V(\tilde f),0)$ is irreducible and $k_{V(\tilde f)}(C(V(\tilde f),0))=k_{V(\tilde g)}(C(V(\tilde g),0))=k$, since $\tilde\varphi$ preserves weakly relative multiplicities. Therefore, we have that $k={\rm ord} (\tilde f)=k\cdot m(C(V(\tilde f),0),0)$. Thus, $m(C(V(\tilde f),0),0)=1$ and, then, $C(V(\tilde f),0)$ is smooth at $0$, which is a contradiction, since $C(V(\tilde f),0)$ cannot be a linear subspace. Therefore, ${\rm ord} (f)\geq {\rm ord} (g)$ and by similar arguments as in the proof of Theorem \ref{main_thm}, we obtain also ${\rm ord} (f)\leq {\rm ord} (g)$.
\end{proof}

As a direct consequence, we obtain Theorem \ref{relative_cor}.

Let us finish this Subsection by showing that Question B$_{Rel}$ has a negative answer and that Question B' also has a negative answer, when we do not require that $C(V(f),0)$ or $C(V(g),0)$ need to be a linear subspace. Here, as it already was remarked in the Introduction, we cannot use A'Campo-L\^e's Theorem as it was used in \cite{Bartolo:2010}.
\begin{proposition}\label{ex_rel_inv}
There exists a family $\{V_t\}$ of hypersurface singularities topologically equisingular such that $C(V(f_t),0)$ is not homeomorphic to $C(V(f_0),0)$ for any $t\not=0$. Moreover, this same family gives a negative answer to Question B$_{Rel}$.
\end{proposition}
\begin{proof}
In the paper \cite{Bartolo:2010} the authors showed that the family $\{V_t\}$ of hypersurface singularities defined as zero locus of
$$
f_t = z^{12} + zy^3x + ty^2x^3 + x^6 + y^5
$$
is topologically equisingular. In fact, they showed that there exists a family of homeomorphisms $\{\varphi_t\colon (\mathbb{C}^3,0)\to (\mathbb{C}^3,0)\}$ satisfying $f_t=f_0\circ \varphi_t$. 

However, ${\bf in}(f_t)=zy^3x+ty^2x^3+y^5=y^2(xyz+tx^3+y^3)$ for $t\not=0$ and ${\bf in}(f_0)=y^3(xz+y^2)$.
Thus, $m(C(V(f_t),0),0)=4$ for $t\not=0$ and $m(C(V(f_0),0),0)=3$. Therefore, by Proposition 3.5 in \cite{BobadillaFS:2017}, $C(V(f_t),0)$ is not homeomorphic to $C(V(f_0),0)$ for any $t\not=0$, which finish the first part of the proof.

Moreover, as it was said by J. Fern\'andez de Bobadilla in a private talk with the author, that family gives also a negative answer to Question B$_{Rel}$. In fact,  
it is clear to see that $k_{V(f_t)}(V(y))=2$ and $k_{V(f_t)}(V(xyz+tx^3+y^3))=1$ for $t\not=0$, but $k_{V(f_0)}(V(y))=3$ and $k_{V(f_0)}(V(xz+y^2))=1$.

Thus, 
$$k_{V(f_0)}(V(y))\not=k_{V(f_t)}(V(y))$$
and 
$$k_{V(f_0)}(V(y))\not=k_{V(f_t)}(V(xyz+tx^3+y^3)),$$
showing that $V(f_t)$ and $V(f_0)$ have different relative multiplicities, when $t\not=0$ and this gives a negative answer to Question B$_{Rel}$.
\end{proof}

\subsection{Lipschitz submanifold vs. parametric Lipschitz manifold}

\begin{proof}[Proof of Proposition \ref{ex_lip}]
For each odd natural number $k$, let $f\colon \mathbb{C}^{k+1}\to \mathbb{C}$ be the polynomial given by $f_k(x_1,...,x_k,z)=x_1^2+...+x_k^2-z^3$. We have that $V(f_k)$ has a unique singularity at $0\in \mathbb{C}^{k+1}$. Moreover, by Proposition in (\cite{Milnor:1968}, Theorem 2.10), there exist $\varepsilon >0$ and a homeomorphism $h\colon (B_{\varepsilon },V(f_k)\cap B_{\varepsilon })\to (B_{\varepsilon },CL_{V(f_k)})$ such that $h(0)=0$ and $V(f_k)\cap \mathbb{S}^{2k+1}_{\varepsilon }$ is a compact smooth manifold, where 
$$B_{\varepsilon } =\{v\in \mathbb{C}^{k+1};\|v\|\leq\varepsilon \},$$ 
$$\mathbb{S}^{2k+1}_{\varepsilon } =\{v\in \mathbb{C}^{k+1};\|v\|=\varepsilon \}$$ 
and 
$$CL_{V(f_k)}=\{tv;t\in [0,1]\mbox{ and }v\in V(f_k)\cap \mathbb{S}^{2k+1}_{\varepsilon }\}.$$
However, it was proven in \cite{Brieskorn:1966} that there is a diffeomorphism $\phi\colon \mathbb{S}^{2k-1}\to V(f_k)\cap \mathbb{S}^{2k+1}_{\varepsilon }$. Since $V(f_k)\cap \mathbb{S}^{2k+1}_{\varepsilon }$ and $\mathbb{S}^{2k-1}$ are compact smooth manifolds, we have that $\phi$ is bi-Lipschitz homeomorphism. Therefore, the mapping $\varphi\colon \mathbb{R}^{2k}\to C_k$ given by 
$$
\varphi(x)=\left \{\begin{array}{ll}
                    \|x\|\cdot\phi(\frac{x}{\|x\|}),& \mbox{ if } x\not =0\\
                    0,& \mbox{ if } x =0
                    \end{array}\right.
$$
is a bi-Lipschitz homeomorphism as well, where $C_k =\{tv;t\geq 0\mbox{ and }v\in V(f_k)\cap \mathbb{S}^{2k+1}_{\varepsilon }\}$.  In particular, $C_k$ is a parametric Lipschitz manifold.\\

\noindent{\bf Claim 1.} $C_k$ is not a Lipschitz submanifold.\\

Suppose that $C_k$ is a Lipschitz submanifold. Then, since $0\in C_k$, there exist an open neighborhood $U\subset \mathbb{C}^{k+1}\equiv \mathbb{R}^{2k+2}$ and a bi-Lipschitz homeomorphism $\phi\colon U\to \phi (U)\subset \mathbb{C}^{k+1}$ such that $\phi(0)=0$ and $\phi(C_k\cap U)=V(z)\cap \phi (U)$, 
where $V(z)=\{(x_1,...,x_k,z)\in \mathbb{C}^{k+1};z=0\}$. Therefore, we have a homeomorphism $\phi\colon (\mathbb{C}^{k+1}, V(z),0)\to (\mathbb{C}^{k+1},V(f_k),0)$. Since $V(z)$ is smooth at $0$, by A'Campo-L\^e's Theorem, $V(f_k)$ is smooth at $0$, which is a contradiction. Thus, Claim 1 is proven and we finish the proof.
\end{proof}

We would like to remark that we cannot obtain counterexamples, like in the proof of Proposition \ref{ex_lip}, being complex analytic sets, since it was proven by the author in \cite[Theorem 4.2]{Sampaio:2016} (see also Theorem 3.1 in \cite{BirbrairFLS:2016}) that if a complex analytic set $X$ is a parametric Lipschitz manifold then $X$ is smooth and, in particular, it is a Lipschitz submanifold.

\subsection{Invariance of the degree by right equivalence}\label{subsec:degree}

Let $f\colon\mathbb{C}^n\to \mathbb{C}$ be a complex polynomial with degree $d>0$. We write
$$f=f_m+f_{m+1}+\cdots+f_{d-1}+f_d$$ 
where each $f_k$ is a homogeneous polynomial of degree $k$. Then, the tangent cone of $V(f)$ at infinity is $C_{\infty}(V(f))=V(f_d)$.

\begin{definition}\label{def_tg_hom_infinity}
For each $n\in \mathbb{N}$, let $\mathcal{H}_{n,\infty}$ be the collection of all germs of homeomorphisms at infinity from $\mathbb{C}^n$ to $\mathbb{C}^n$. Let $\mathcal{T}_{\mathcal{H},\infty}\subset \mathcal{H}_{\infty}:=\bigcup\limits_{n\in\mathbb{N}}\mathcal{H}_{n,\infty}$ be the maximal subset with respect the inclusion satisfying:
\begin{itemize}
\item [i)] If $\varphi\in \mathcal{T}_{\mathcal{H},\infty}$, then $\varphi\times {\rm id}_{\mathbb{C}}\in \mathcal{T}_{\mathcal{H},\infty}$;
\item [ii)] If $\varphi\in \mathcal{T}_{\mathcal{H},\infty}$ and $f,g\colon \mathbb{C}^n\to \mathbb{C}$ are complex polynomials such that $\varphi(V(f))=V(g)$ (outside of a compact subset), then there exists a homeomorphism $\phi\colon C_{\infty}(V(f))\to C_{\infty}(V(g))$.
\end{itemize}
\end{definition}
\begin{example}\label{exam_lip}
If $Lip_{n,\infty}$ is the collection of all germs of bi-Lipschitz homeomorphisms at infinity from $\mathbb{C}^n$ to $\mathbb{C}^n$, then by Theorem 4.5 in \cite{FernandesS:2018}, it is easy to verify that $Lip_{n,\infty}\subset \mathcal{T}_{\mathcal{H},\infty}$.
\end{example}

As an application of the proof of Theorem \ref{main_thm}, we obtain a result about invariance of the degree.
\begin{corollary}\label{right_infinty}
Let $f,g\colon \mathbb{C}^n\to \mathbb{C}$ be complex polynomials. If there exists $\varphi\in \mathcal{T}_{\mathcal{H},\infty}$ such that $f=g\circ\varphi$ (outside of a compact subset), then ${\rm deg} (f)={\rm deg} (g)$.
\end{corollary}
\begin{proof}
Suppose that ${\rm deg} (f)>{\rm deg} (g)$. If $k={\rm deg} (f)$, then we have that $k>1$, since $g$ cannot be constant. Let $f_k$ be the homogeneous polynomial formed by the monomials of $f$ that have degree $k$. Let us define $\tilde f, \tilde g\colon (\mathbb{C}^{n+1},0)\to (\mathbb{C},0)$ by
$\tilde f(x,t)=f(x)+t^k$ and $\tilde g(x,t)=g(x)+t^k$. Then, $C(V(\tilde g),0)=\{t=0\}=\mathbb{C}^n\times\{0\}$. Moreover, since $f_k\not \equiv 0$ and $k>1$, we get that $f_k+t^k$ cannot be a power of a linear form and, in particular, $C(V(\tilde f),0)=V(f_k+t^k)$ is not a linear subspace.
Now, let us define $\tilde\varphi\colon (\mathbb{C}^{n+1},0)\to (\mathbb{C}^{n+1},0)$ by $\tilde\varphi(x,t)=(\varphi(x),t)$. By item i) of the definition \ref{def_tg_hom_infinity}, we have that $\tilde\varphi\in \mathcal{T}_{\mathcal{H},\infty}$. By item iii) of the definition \ref{def_tg_hom_infinity}, there exists a homeomorphism $\tilde\phi\colon (V(f_k+t^k),0)\to (\mathbb{C}^n\times\{0\},0)$. Thus, by Prill's Theorem (Theorem in \cite{Prill:1967}), $V(f_k+t^k)$ is a linear subspace, which is a contradiction. Therefore, ${\rm ord} (f)\leq {\rm ord} (g)$ and by similar arguments as in the proof of Theorem \ref{main_thm}, we obtain also ${\rm ord} (f)\geq {\rm ord} (g)$.
\end{proof}

We would like to say that in general it is hard the degree to be preserved for some equivalence, as we can see in next example.
\begin{example}
Let $f,g\colon \mathbb{C}^n\to \mathbb{C}$ be two complex polynomials such that $f(x,y)=y-x^2$ and $g(x,y)=y$. Let $\varphi\colon \mathbb{C}^2\to \mathbb{C}^2$ be the polynomial diffeomorphism given by $\varphi(x,y)=(x,y-x^2)$. Then $f=g\circ \varphi$. However, ${\rm deg} (f)=2$ and ${\rm deg} (g)=1$. 
\end{example}
In particular, a polynomial diffeomorphism does not need to belong to $\mathcal{T}_{\mathcal{H},\infty}$.

Thus, by Example \ref{exam_lip} and Corollary \ref{right_infinty}, we obtain the following particular case of Corollary 3.15 in \cite{Sampaio:2019}.
\begin{corollary}
Let $f,g\colon\mathbb C^n\to \mathbb C$ be two polynomials. If there exists $\varphi \in Lip_{n,\infty}$ such that $f=g\circ\varphi$ (outside of a compact subset) then ${\rm deg}(f)={\rm deg}(g)$.
\end{corollary}

\section{Homeomorphism between two Nash Modifications}\label{sec:Nash}
In this Section we show that Question B$_{Nash}$ has a negative answer when we consider real analytic sets. 
In order to do this, we need of some preliminaries.

Let $X\subset\mathbb{R}^m$ be a path connected subset. Let us consider the following distance on $X$:  given two points $x_1,x_2\in X$, $d_X(x_1,x_2)$  is the infimum of the lengths of paths on $X$ connecting $x_1$ to $x_2$. Let us observe that
$$ \| x_1 - x_2 \| \leq d_X(x_1,x_2), \quad \forall \ x_1,x_2\in X.$$

\begin{definition}\label{definition-ne}
A subset $X\subset\mathbb{R}^m$ is called \emph{{\bf Lipschitz normally embedded (LNE)}} if there exists a constant $k\geq 1$ such that
$$d_X(x_1 , x_2) \leq  k \| x_1 - x_2 \|, \quad \forall \ x_1,x_2\in X.$$
\end{definition}
\begin{proposition}\label{ne_cone}
Let $L\subset\mathbb{S}^k$ be a set and ${\rm Cone}(L):=\{tx;\, x\in L$ and $t\in [0,+\infty )\}$. If $L$ is a LNE set then ${\rm Cone}(L)$ is a LNE set, as well.
\end{proposition}
Let us remark that Proposition \ref{ne_cone} already was proven in (\cite{KernerPR:2018}, Proposition 2.8 (a)). However, in his thesis (\cite{Sampaio:2015}, Example 4.3.7), the author of this paper proved Proposition \ref{ne_cone} for the case of a surface in $\mathbb{R}^3$ with isolated singularity and as it was remarked in \cite{KernerPR:2018}, it is possible to generalize the proof in \cite{Sampaio:2015} to arbitrary dimensions and, thus, we present this generalization below. 
\begin{proof}[Proof of Proposition \ref{ne_cone}]
Since $L$ is LNE, there exists a constant $k\geq 1$ such that
$$d_L(x_1 , x_2) \leq  k \| x_1 - x_2 \|, \quad \forall \ x_1,x_2\in L.$$
It is enough to show that there exists a positive constant $C$ such that $d_V(x,y)\leq C\|x-y\|$ for all $x,y\in V$.
Given $x,y\in V$, we can suppose that $\|y\|\geq \|x\|$. If $x=0$, then $d_V(0,y)=\|0-y\|=\|y\|$, since $V$ is a real cone with vertex at origin. If $x\not=0$ and again using the fact that $V$ is a real cone with vertex at origin, we have $\frac{x}{\|x\|},\frac{y}{\|y\|}\in L$. Thus, given $\varepsilon>0$, there exists a Lipschitz curve $\alpha\colon [0,1]\to L$ such that $\alpha(0)=\frac{x}{\|x\|}$, $\alpha(1)=\frac{y}{\|y\|}$  and 
$$ length(\alpha)=\int_0^1\|\alpha'(t)\|dt \leq d_L(\textstyle{\frac{x}{\|x\|},\frac{y}{\|y\|})+\varepsilon\left\|\frac{x}{\|x\|}-\frac{y}{\|y\|}\right\|\leq (k+\varepsilon)\left\|\frac{x}{\|x\|}-\frac{y}{\|y\|}\right\|}.$$ 
Then, we define $\gamma\colon [0,1]\to V$ given by $\gamma(t)=(t\|y\|+(1-t)\|x\|)\alpha(t)$. It is clear that $\gamma$ is a Lipschitz curve satisfying $\gamma(0)=x$ and $\gamma(1)=y$. Hence,
\begin{eqnarray*}
 d_V(x,y) & \leq & length(\gamma) = \int_0^1\|\gamma'(t)\|dt\\
					    & = & \int_0^1\|(\|y\|-\|x\|)\alpha(t)+(t\|y\|+(1-t)\|x\|)\alpha'(t)\|dt\\
					    & \leq & \int_0^1\big|\|y\|-\|x\|\big|\cdot\|\alpha(t)\|dt+\int_0^1(t\|y\|+(1-t)\|x\|)\|\alpha'(t)\|dt\\
					    & \leq & \big|\|y\|-\|x\|\big|+\|y\|\int_0^1\|\alpha'(t)\|dt \mbox{, since } \|y\|\geq \|x\|\\
					    & \leq & \|x-y\|+\|y\|(k+\varepsilon)\left\|\textstyle \frac{x}{\|x\|}-\frac{y}{\|y\|}\right\|\\
					    & \leq & \|x-y\|+(k+\varepsilon)\big\|\textstyle \frac{\|y\|}{\|x\|}x-y\big\|\\
					    & \leq & \|x-y\|+(k+\varepsilon)\left\|\textstyle \|y\|\frac{x}{\|x\|}-\|x\|\frac{x}{\|x\|}\right\|+(k+\varepsilon)\| x-y\|\\
					    & = & (k+\varepsilon)\big|\|y\|-\|x\|\big|+(k+\varepsilon+1)\| x-y\|\\
					    & \leq & (2(k+\varepsilon)+1)\| x-y\|.
\end{eqnarray*}
Therefore, by taking the limit as $\varepsilon\to 0$, we obtain that $d_V(x,y)\leq C\| x-y\|$ for all $x,y\in V$, where $C=2k+1$.
\end{proof}
\begin{definition}\label{inner_lipschitz function}
Let $X\subset\mathbb{R}^n$ and $Y\subset\mathbb{R}^m$. A mapping $f\colon X\rightarrow Y$ is called {\bf inner Lipschitz} if there exists $\lambda >0$ such that is
$$d_Y(f(x_1),f(x_2))\le \lambda d_X(x_1,x_2), \quad \mbox{for all }x_1,x_2\in X.$$ A Lipschitz mapping $f\colon X\rightarrow Y$ is called
{\bf inner bi-Lipschitz} if its inverse mapping exists and is inner Lipschitz. When there exists an inner bi-Lipschitz mapping $f\colon X\rightarrow Y$ we say that $X$ and $Y$ are inner bi-Lipschitz homeomorphic.
\end{definition}

\begin{proof}[Proof of Proposition \ref{ex_Nash}]
Let us consider $V=\{(x,y,z)\in\mathbb{R}^3;\, z^{2019}=x^{2019}+y^{2019}\}$. 
Since the link of $V$, $L:=V\cap \mathbb{S}^2$, is a smooth submanifold of $\mathbb{R}^3$, connected and compact, then $L$ is LNE. By Proposition \ref{ne_cone}, $V$ is LNE.

Therefore, by Theorem 8.3 in \cite{Birbrair:2008}, there exists $\beta\geq 1$ such that $(V,0)$ and $(X_{\beta},0)$ are inner bi-Lipschitz homeomorphic, where $X_{\beta}=\{(x,y,z)\in\mathbb{R}^3;\, x^2+y^2=z^{2\beta}$ and $z\geq 0\}$. Since $(V,0)$ and $(X_{\beta},0)$ are LNE sets, then $(V,0)$ and $(X_{\beta},0)$ are also bi-Lipschitz homeomorphic. In addiction, $X_{\beta}$ has zero density at the origin if and only if $\beta> 1$. However, the tangent cone of $V$ at origin is itself and, in particular, $V$ does not have zero density at the origin, then $\beta=1$, since an inner bi-Lipschitz homeomorphism preserves zero density (see Proposition 2.3 in \cite{BirbrairFN:2010} and Th\'eor\`eme 3.8 in \cite{KurdykaR:1989}). Moreover, $(X_{1},0)$ and $(\mathbb{R}^2,0)$ are bi-Lipschitz homeomorphic, since $X_{1}$ is the graph of the Lipschitz mapping $f\colon (\mathbb{R}^2,0)\to (\mathbb{R},0)$ given by $f(x,y)=(x^2+y^2)^{\frac{1}{2}}$.  Therefore, $(V,0)$ and $(\mathbb{R}^2,0)$ are bi-Lipschitz homeomorphic, as well.

Now, it is easy to see that there is no homeomorphism from $V'_{Nash}$ to ${\mathbb{R}^2}'_{Nash}$ sending $\eta_V^{-1}({\rm Sing}(V))$ over $\eta_{\mathbb{R}^2}^{-1}({\rm Sing}(\mathbb{R}^2))$, since $\eta_V^{-1}({\rm Sing}(V))\not = \emptyset$.
\end{proof}
In particular, the above proposition says that to give a positive answer to Question B$_{Nash}$ in the case of complex analytic sets, we really need to use the complex structure of the sets.

\noindent {\bf Acknowledgements.}
The author would like to thank Javier Fern\'andez de Bobadilla for telling how to get a negative answer to Question B$_{Rel}$ and Alexandre Fernandes for the discussions about Question B$_{Nash}$. The author would also like to thank the anonymous referee for his useful comments.

% BibTeX users please use one of
%\bibliographystyle{spbasic}      % basic style, author-year citations
%\bibliographystyle{spmpsci}      % mathematics and physical sciences
%\bibliographystyle{spphys}       % APS-like style for physics
%\bibliography{}   % name your BibTeX data base

\begin{thebibliography}{99}
\bibitem{Acampo:1973} 
{A'Campo, N.}
{\it Le nombre de Lefschetz d'une monodromie}.
(French) Nederl. Akad. Wetensch. Proc. Ser. A 76 = Indag. Math., vol. 35 (1973), 113--118.

\bibitem{Bartolo:2010}
{Bartolo, E.A.;  Fern\'andez-Bobadilla, J.; Luengo, I. and Melle-Hern\'andez, A.}
{\it Milnor Number of Weighted-L\^e-Yomdin Singularities}.
International Mathematics Research Notices, vol. 2010 (2010), no. 22, 4301--4318.

\bibitem{Birbrair:2008}
{Birbrair, L.}
{\em Lipschitz geometry of curves and surfaces definable in o-minimal structures}.
Illinois J. Math., vol. 52 (2008), no. 4, 1325--1353.

% \bibitem{BirbrairFG:2012}
% {Birbrair, L.; Fernandes, A. and Grandjean, V.}
% {\it Collapsing topology of isolated singularities}.
% arXiv:1208.4328v1 [math.MG], 2012.

% \bibitem{BirbrairFG:2017}
% {Birbrair, L.; Fernandes, A. and Grandjean, V.}
% {\it Thin-thick decomposition for real definable isolated singularities}.

\bibitem{BirbrairFN:2010}
Birbrair, L.; Fernandes, A. and Neumann, W. D. 
{\it Separating sets, metric tangent cone and applications for complex algebraic germs.}
Selecta Math. (N.S.) 16 (2010), no. 3, 377--391.
% Indiana University Math. J.,  vol. 66 (2017), 547--557.

\bibitem{BirbrairFLS:2016}
{Birbrair, L.; Fernandes, A.; L\^e D. T. and Sampaio, J. E.}
{\em Lipschitz regular complex algebraic sets are smooth}.
Proceedings of the American Mathematical Society, vol. 144 (2016), no. 3, 983--987.

\bibitem{BirbrairFSV:2018}
{Birbrair, L.; Fernandes, A.; Sampaio, J. Edson and Verbitsky, M.}
{\it Multiplicity of singularities is not a bi-Lipschitz invariant}.
arXiv:1801.06849v1 [math.AG], preprint (2018).


\bibitem{BirbrairG:2018}
{Birbrair, L. and Gabrielov, A.}
{\it Ambient Lipschitz Equivalence of Real Surface Singularities}.
International Mathematics Research Notices, vol. 2019 (2019), no. 20, 6347–6361. % https://doi.org/10.1093/imrn/rnx328.


\bibitem{BondilL:2002} 
{Bondil, R. and L\^e D. T.}
{\it R\'esolution des singularit\'es de surfaces par \'eclatements normalis\'es}, 
in Trends in Singularities, 31--81, ed. by A. Libgober and M. Tibar, Birkh\"auser Verlag, 2002.

% \bibitem{BrasseletFGR:2012}
% {Brasselet, J.-P.; Fernandes, A.; Grulha Jr, N. G. and Ruas, M. A. S.}
% The Nash modifications and the bi-Lipschitz equivalence
% arXiv:1206.5153v1 [math.AG], preprint (2012).

\bibitem{Brieskorn:1966} 
{Brieskorn, E. V.} 
\newblock {\em Examples of singular normal complex spaces which are topological manifolds}. 
\newblock Proc. Nat. Acad. Sci. U.S.A., vol. 55 (1966), 1395--1397.

% \bibitem{Chirka:1989} 
% {Chirka, E.M. }
% {\it Complex analytic sets}. 
% Kluwer Academic Publishers, Dordrecht, 1989.

\bibitem{Dimca:1992} 
{Dimca, A. }
{\it Singularities and Topology of Hypersurfaces}. 
Springer-Verlag, New York, 1992.

% \bibitem{Ephraim:1976a} 
% {Ephraim, R.}
% {\it $C^1$ preservation of multiplicity}. 
% Duke Math., vol. 43 (1976), 797--803.

\bibitem{Eyral:2007} 
{Eyral, C.}
\newblock {\em Zariski's multiplicity questions - A survey}.
\newblock New Zealand Journal of Mathematics, vol. 36 (2007), 253--276.

\bibitem{FernandesS:2016}
{Fernandes, A. and Sampaio, J. Edson}.
{\it Multiplicity of analytic hypersurface singularities under bi-Lipschitz homeomorphisms}.
Journal of Topology, vol. 9 (2016), 927--933.

\bibitem{FernandesS:2018}
{Fernandes, A. and Sampaio, J. E.}
{\it On Lipschitz rigidity of complex analytic sets}.
J. Geom. Anal., vol. online (2019), 1--13. \url{https://doi.org/10.1007/s12220-019-00162-x}.

\bibitem{Bobadilla:2005} 
{Fern\'andez de Bobadilla, J.}. 
{\it Answers to some equisingularity questions.}
Invent. math., vol. 161 (2005), 657--675.

\bibitem{BobadillaFS:2017}
{Fern\'andez de Bobadilla, J.; Fernandes, A. and Sampaio, J. Edson}
{\it Multiplicity and degree as bi-Lipschitz invariants for complex sets}.
Jornal of Topology, vol. 11 (2018), 957--965.

% \bibitem{BobadillaP:2008} 
% {Fern\'andez de Bobadilla, J. and Pereira, M. Pe}. 
% {\it Equisingularity at the normalisation.}
% Journal of Topology, vol. 1 (2008), 879--909.

% \bibitem{Gau and Lipman:1983}
% {Gau, Y.-N. and Lipman, J.}
% {\it Differential invariance of multiplicity on analytic varieties}.
% Inventiones mathematicae, vol. 73 (1983), no. 2, 165--188.

\bibitem{KernerPR:2018}
{Kerner, D.; Pedersen, H. M. and Ruas, M. A. S.}
{\it Lipschitz normal embeddings in the space of matrices}
Math. Z. (2018) 290: 485. \url{https://doi.org/10.1007/s00209-017-2027-4}

\bibitem{King:1978}
{King, H. C.}
{\it Topological Type of Isolated Critical Points}.
Annals of Mathematics, Second Series, vol. 107 (1978), no. 3, 385--397.

\bibitem{KurdykaR:1989} 
{Kurdyka, K. and Raby, G.}
{\it Densit\'e des ensembles sous-analytiques}. 
Ann. Inst. Fourier (Grenoble), vol. 39 (1989), no.3, 753--771.

\bibitem{Le:1973} 
{L\^e D. T.}
{\it Calcul du nombre de cycles \'evanouissants d'une hypersurface complexe}. 
(French) Ann. Inst. Fourier (Grenoble), vol. 23 (1973), no. 4, 261--270.

\bibitem{Le:2000}
{L\^e D. T.}
{\it Geometry of surface singularities}, 
in Singularities (Sapporo, 1998), 163--180, Adv. Stud. in Pure Math., vol. 20, 2000. 

\bibitem{Milnor:1968}
{Milnor, J.}
{\it Singular points of complex hypersurfaces}.
Princeton University Press, Princeton, 1968.

\bibitem{NaumannS:2011} 
{Naumann, J. and Simader, C. G.}
{\it Measure and Integration on Lipschitz-Manifolds}. 
Preprints aus dem Institut f\"ur Mathematik - 15, Mathematik-Preprints, 2011, 43 p. \url{https://doi.org/10.18452/2773}.

% \bibitem{Pawlucki:1985}
% Paw\l ucki, Wies\l aw
% {\it Quasi-regular boundary and Stokes' formula for a sub-analytic leaf}.
% In: \L awrynowicz J. (eds) Seminar on Deformations. Lecture Notes in Math., vol 1165. Berlin, Heidelberg: Springer. 1985.

\bibitem{Prill:1967} 
{Prill, D. }
{\em  Cones in complex affine space are topologically singular}.
Proc. of AMS, vol. 18 (1967), 178--182.

\bibitem{Saeki:1989}
{Saeki, O.}
{\it Topological types of complex isolated hypersurface singularities}.
Kodai Math. J., vol. 12 (1989), 23--29.

\bibitem{Sampaio:2015}
{Sampaio, J. E.} 
{\it Regularidade lipschitz, invari\^ancia da multiplicidade e a geometria dos cones tangentes
de conjuntos anal\'itico}. 
Ph.D. thesis, Universidade Federal Do Cear\'a (2015). \url{http://www.repositorio.ufc.br/bitstream/riufc/12545/1/2015_tese_jesampaio.pdf}

\bibitem{Sampaio:2016}
{Sampaio, J. Edson}.
{\it Bi-Lipschitz homeomorphic subanalytic sets have bi-Lipschitz homeomorphic tangent cones}.
Selecta Mathematica: New Series, vol. 22 (2016), no. 2, 553--559.

\bibitem{Sampaio:2017}
{Sampaio, J. Edson}.
{\it Multiplicity, regularity and blow-spherical equivalence of complex analytic sets}.
arXiv:1702.06213 [math.AG], preprint (2017).

\bibitem{Sampaio:2019}
{Sampaio, J. Edson}.
{\it On Zariski's multiplicity problem at infinity}.
Proc. Amer. Math. Soc., vol. 147 (2019), 1367--1376. \url{https://doi.org/10.1090/proc/14351}

\bibitem{Shapiro:1990}
{Shapiro, A.}
{\it On Concepts of Directional Differentiability}.
Journal of optimization theory and applications, vol. 66 (1990), no. 3, 477--487.

\bibitem{Spivakovsky:1990}  
{Spivakovsky, M.}
{\it Sandwiched singularities and desingularization of surfaces by normalized Nash transformations}.
Ann. of Math., vol. 131 (1990), 411--491.

% \bibitem{Trotman:1977}
% {Trotman, D.}
% {\it Multiplicity is a {C}$^1$ invariant}.
% University Paris 11 (Orsay), Preprint, 1977.

\bibitem{user103279}
user103279 (https://mathoverflow.net/users/103279/user103279).
{\it Topological invariance of the relative multiplicities}, 
URL (version: 2017-01-13): \url{https://mathoverflow.net/q/259520}.

\bibitem{Valette:2010} 
{Valette, G.}. 
{\it Multiplicity mod 2 as a metric invariant}.  
Discrete Comput. Geom., vol. 43 (2010), 663--679.

\bibitem{Zariski:1932}
{Zariski, O.}
{\it On the topology of algebroid singularities.}
Amer. J. Math., vol. 54 (1932), 453--465.

\bibitem{Zariski:1939}
{Zariski, O.}
{\it The reduction of the singularities of an algebraic surface}.
Ann. of Math., vol. 40 (1939), 639--689. 

\bibitem{Zariski:1971} 
{Zariski, O.}
{\it Some open questions in the theory of singularities}. 
Bull. of the Amer. Math. Soc., vol. 77 (1971), no. 4, 481--491.
\end{thebibliography}
%    Bibliographies can be prepared with BibTeX using amsplain,
%    amsalpha, or (for "historical" overviews) natbib style.
% \bibliographystyle{amsalpha}
% \begin{thebibliography}{A}

% Non-BibTeX users please use

\end{document}